\newif\ifpictures
\picturestrue 
\documentclass[12pt]{amsart}
\usepackage{amssymb}
\usepackage{amsmath}
\usepackage{dsfont}
\usepackage{amscd}
\usepackage{blkarray}
\usepackage[pdftex]{graphicx}
\usepackage[normalem]{ulem}
\usepackage{hyperref}
\usepackage{tikz}
\usepackage{xcolor}
\usepackage{bbm} 
\usepackage{amssymb}
\usepackage{amsmath}
\usepackage{youngtab}
\usepackage{young}
\usepackage{xfrac}
\usepackage{enumitem}
\usepackage{float}
\usepackage{array}
\usepackage{ytableau}
\usepackage{amsfonts,mathrsfs}
\usepackage{bbm,amsmath,amssymb,amsfonts,graphicx,color,subfig,mathtools, verbatim}

\ifpictures

\fi
\usepackage{mathrsfs}
\usetikzlibrary{arrows}
\usepackage{siunitx}
\newcommand{\fonc}[5]{\begin{array}{ccccc}
#1 & : & #2 & \to & #3 \\
 & & #4 & \mapsto & #5 \\
\end{array}}

\numberwithin{equation}{section}
\newtheorem{thm}{Theorem}
\newtheorem{prop}[thm]{Proposition}

\newtheorem{cor}[thm]{Corollary}

\theoremstyle{definition}

\newtheorem{example}[thm]{Example}
\newtheorem{remark1}[thm]{Remark}
\newtheorem{openproblem1}[thm]{Open problem}

\DeclareMathOperator{\len}{len}

\newenvironment{rem}{\begin{remark1}\rm}{\end{remark1}}

\numberwithin{thm}{section}

\newcounter{FNC}[page]
\def\newfootnote#1{{\addtocounter{FNC}{2}$^\fnsymbol{FNC}$%
		\let\thefootnote\relax\footnotetext{$^\fnsymbol{FNC}$#1}}}

\newcommand{\N}{\mathbb{N}}

\newcommand{\R}{\mathbb{R}}

\newcommand{\hbeta}{\hat{\beta}}

\newcommand{\sage}{\mathrm{SAGE}}

\newcommand{\cA}{\mathcal{A}}
\newcommand{\cX}{\mathcal{T}}
\newcommand{\cB}{\mathcal{B}}

\DeclareMathOperator{\conv}{conv}

\DeclareMathOperator{\wt}{wt}

\DeclareMathOperator{\Stab}{Stab}

\newcommand{\sym}{\mathcal{S}}

\newcommand{\set}[1]{\{#1\}}

\usepackage{todonotes}

\usepackage{xcolor}

\title{Symmetry reduction in AM/GM-based optimization}

\author{Philippe Moustrou}

\author{Helen Naumann}

\author{Cordian Riener}

\author{Thorsten Theobald}

\author{Hugues Verdure}

\address{Philippe Moustrou:
  Institut de Math\'ematiques de Toulouse, UMR 5219,  UT2J, 
  31058 Toulouse, France}

\address{Helen Naumann, Thorsten Theobald:
        FB 12 -- Institut f\"ur Mathematik, Goethe-Uni\-ver\-si\-t\"at,
        Postfach 11 19 32, 60054 Frankfurt am Main, Germany}

\address{Cordian Riener, Hugues Verdure:
  Department of Mathematics and Statistics, UiT -- The Arctic University of Norway,
  9037 Troms\o, Norway}

\date{\today}


\begin{document}
	\begin{abstract} 
		The arithmetic mean/geometric mean-inequality (AM/GM-inequality) facilitates 
		classes of non-negativity certificates and of relaxation techniques for polynomials and,
		more generally, for exponential sums. 
		Here, we present a first systematic study of the AM/GM-based techniques  
		in the presence of symmetries under the linear action of a finite group.
		We prove a symmetry-adapted representation theorem and 
		develop techniques
		to reduce the size of the resulting relative entropy programs. 
		We study in more detail the complexity gain in the
                case of the symmetric group. In this setup, we can
                show in particular certain stabilization results.
		We exhibit several sequences of examples in growing dimensions where the size of the reduced problem stabilizes.
		Finally, we provide some numerical results, emphasizing the computational speed-up.
	\end{abstract}

	\maketitle
	
	\section{Introduction}

	Deciding whether a real function only takes non-negative values is a fundamental question in real algebraic geometry.
	Non-negativity certificates and optimization approaches
	are tightly related to each other
	by observing that the infimum $f^*$ of a function ${f: \R^n \to\R}$
	can be expressed as the largest $\lambda \in \R$ for which 
	$f-\lambda$ is non-negative on $\R^n$:
	\[
	f^*=\inf \{ f(x) \ : \ x \in \R^n \} \ = \ 
	\sup \{ \lambda \in \R \, : \ f - \lambda \text{ is non-negative on } \R^n \}.
	\]

	Both in the context of polynomials and in the broader context of exponential sums, 
	the last years have seen strong interest in non-negativity certificates 
	and optimization techniques based on the arithmetic mean/geometric mean-inequality
	(AM/GM inequality).
	More precisely, an exponential
	sum (or \emph{signomial}) 
	supported on a finite subset $\mathcal{T}\subset  \R^n$ 
	is a linear combination
	$\sum_{\alpha \in \cX} c_\alpha \exp(\langle \alpha, x\rangle)$
	with real coefficients $c_{\alpha}$.
	In particular cases, the non-negativity of the real function defined by an exponential sum can be decided via the arithmetic-geometric mean inequality.
	For example, for support points $\alpha_0, \ldots, \alpha_m
	\in \R^n$
	and coefficients $\lambda = (\lambda_1, \ldots, \lambda_m) \in \R_{+}^{n}$ 
	satisfying $\sum_{i=1}^m \lambda_i = 1$ 
	and $\sum_{i=1}^m \lambda_i \alpha_i= \alpha_0$, the exponential sum
	\[
	\sum_{i=1}^m \lambda_i \exp({\langle \alpha_i, x}\rangle) - \exp({\langle\alpha_0, x\rangle}) 
	\]
	is non-negative on $\R^n$ as a consequence of the weighted 
	arithmetic-geometric mean inequality, namely
	$\sum_{i=1}^m \lambda_i \exp({\langle \alpha_i, x\rangle}) \geqslant  
	\prod_{i=1}^m (\exp (\langle{\alpha_i, x}\rangle))^{\lambda_i}$. Clearly, sums
	of such exponential sums are non-negative as well.
	Note that exponential sums can be seen as a generalization of polynomials:
	when $\mathcal{T} \subset  \N^n$, the transformation
	$x_ i = \ln y_i$ gives polynomial functions
	$y \mapsto \sum_{\alpha \in \mathcal{T}} c_{\alpha} y^{\alpha}$ 
	on $\R_{>0}^n$.

	These AM/GM-based certificates appear to be particularly useful in sparse settings.
	In the specialized situation of polynomials, they can be seen as an alternative 
	to non-negativity certificates based on sums of squares. 
	The ideas of these approaches go back 
	to Reznick \cite{reznick-1989} and have been recently brought back into
	the focus of the developments by Pantea, Koeppl, and Craciun \cite{Pantea2012},
	Chandrasekaran and Shah
	\cite{chandrasekaran-shah-2016} (``\emph{SAGE}'' cone: 
	\emph{sums of arithmetic-geometric exponentials}) and
	Iliman and de Wolff \cite{iliman-dewolff-resmathsci} (``\emph{SONC}'' cone:
	\emph{sums of non-negative circuit polynomials}), see also \cite{knt-2020}
	for a generalized, uniform framework. 
	The AM/GM certificates can be effectively obtained by relative entropy programming (see 
	\cite{chandrasekaran-shah-2016,chandrasekaran-shah-rel-entropy}), and in restricted settings these relative entropy programs become geometric 
	programs \cite{iliman-de-wolff-2016-siopt}.
	These techniques have been extended to cover constrained 
	situations, prominently by the work of Murray, Chandrasekaran and Wierman
	based on partial dualization \cite{mcw-2019}. This method can also be
	approached from sublinear circuits, see \cite{mnt-2020}.
	Furthermore, in the setting of polynomials, the AM/GM-based approaches can be 
	combined with sums of squares \cite{karaca-2017}. 
	Other recent approaches to sparse polynomials besides the ones based
	on the AM/GM inequality can be found in the sparse moment hierarchies 
	\cite{wml-2020,wml-2019} and in the works exploiting correlative 
        sparsity \cite{lasserre-2006}, \cite{wkk-2006}. Term sparsity is related to sign-symmetries
        and it is possible to combine correlative and term sparsity \cite{wmln-2020}.
	
	From an algebraic point of view, a problem is \emph{symmetric} when it is invariant under some group action.
	Symmetries are ubiquitous in the context of polynomials and optimization, since they manifest both in the problem formulation and the solution set.
	This often allows to reduce the complexity of the corresponding algorithmic questions. 
	Regarding the set of solutions, it was observed by Terquem as early as in 1840 that a symmetric polynomial does not always have a fully symmetric minimizer (see also Waterhouse's survey~\cite{waterhouse-1983}). However, in many instances, the set of minimizers contains highly symmetric points 
	(see \cite{friedl2018reflection, moustrou2019symmetric,riener-2012,timofte-2003}).
	With respect to problem formulations, symmetry reduction has provided essential advances in many situations (see, 
	for example, \cite{bachoc2008new,bazan-hubert-2021,de-klerk-sotirov-2010,dobre-vera-2015}), especially in the context of sums of squares 
	(see \cite{bgsv-2012,blekherman-riener-2020,
	debus2020reflection,gatermann-parrilo-2004, hhs-2020, rsst-2018,rtjl-2013}). 

	\smallskip

	The current paper starts with the question to which extent symmetries can be
	exploited in AM/GM-based optimization assuming that the problem affords symmetries. 
	We provide a first systematic study of the AM/GM-based approaches in 
	$G$-invariant situations under the action of a group $G$. 
	Our focus is on symmetry-adapted representation theorems,
	and algorithmic symmetry reduction
	techniques.

	\medskip
	
	{\bf Our contributions.}
	1. We prove a symmetry-adapted decomposition 
	theorem and
	develop a symmetry-adapted relative entropy formulation of the cone of SAGE exponentials in a general $G$-invariant setting. 
	
	2. This adaption
	reduces the size of the resulting relative entropy programs or geometric
	programs, see Theorem~\ref{th:symmetric-decomp}, 
	Theorem~\ref{co:entropy-symm1} and Corollary~\ref{co:reduce-a}. As revealed by these statements, the
	gain depends on the orbit structure of the group action.
	
	3. In the case of the symmetric group, we use combinatorial aspects of the representation theory of the symmetric group in order to measure the size of the resulting relative entropy program.
	In particular, we identify situations in which the size of the symmetry adapted relative entropy program stabilizes with respect to the number of variables, see 
	Theorem~\ref{thm:stabilization}.

	4. We evaluate the structural results in the paper in terms of computations.
	In situations with strong symmetry structure, the number of variables
	and the number of equations and inequalities becomes substantially smaller.
	Accordingly, the interior-point solvers underlying the computation of SAGE bounds
	then show strong reductions of computation time. In various cases, the 
	symmetry-adapted computation succeeds when the conventional SAGE 
	computation fails.	
	
	We mostly concentrate on the unconstrained optimization, but the techniques
	can generally also be extended to the constrained case. See, for example,
	Corollary~\ref{co:entropy-symm2}. Constrained versions
        of the SAGE techniques are still rather recent and practical implementations 
        in an early stage; see the recent work \cite{dressler-murray-2021} for converging
        hierarchies and their implementation.
		
	\smallskip
	
	The paper is structured as follows. After collecting relevant notions and concepts
	in Section~\ref{se:prelim},
	we provide in Section~\ref{se:symmetrization} a specific way of writing
	sums of arithmetic-geometric exponentials in the presence of a group symmetry.
	In Section~\ref{se:sym-re}, we study how to characterize and to decide whether a $G$-symmetric exponential sum
	is contained in the SAGE cone with reduced relative entropy programs. 
	The case of the symmetric group is treated in Section~\ref{se:Sn}, while Section~\ref{se:computations} provides experimental results of an implementation
	of the symmetry reduction techniques. 
	We conclude the paper in Section~\ref{se:conclusion}.

\medskip	

\noindent
	{\bf Acknowledgement.}
The authors gratefully acknowledge partial support through the
project ``Real Algebraic Geometry and Optimization'' jointly funded by
the German Academic Exchange Service DAAD and
the Research Council of Norway RCN, through the Troms\o\ Research Foundation grant agreement 17matteCR as well as through the project Pure Mathematics in Norway funded by the
Trond Mohn Foundation and by the Troms\o\ Research Foundation. 
Thanks also to Riley Murray as well as to the anonymous referees
for helpful comments.

	\section{Preliminaries\label{se:prelim}}
	
	Throughout the article, we use the notation $\N=\{0,1,2,3,\ldots\}$.
	For a finite subset $\cX\subset \R^n$, let $\R^\cX$ be the set of 
	$|\cX|$-tuples whose components are indexed by the set $\cX$. 
	We denote by $\langle \cdot , \cdot \rangle$ the standard Euclidean inner product in $\R^n$.
	\subsection*{The SAGE cone} For a given non-empty finite set $\mathcal{T}$,
	we consider exponential sums supported on $\mathcal{T}$ as defined in the 
	Introduction.
	For
	finite $\cX\subset \R^n$, the SAGE cone
	$C_{\mathrm{SAGE}}(\cX)$ is defined as
	\[
	C_{\mathrm{SAGE}}(\cX) := \sum_{\beta \in \cX}
	C_{\text{AGE}}(\cX \setminus \{\beta\},\beta),
	\]
	where for $\mathcal{A} := \mathcal{T} \setminus \{\beta\}$
	\[
	C_{\mathrm{AGE}}(\cA,\beta) := \Big\{ f = 
	\sum\limits_{\alpha\in\cA} c_\alpha e^{\langle \alpha, x \rangle}
	+ c_\beta e^{\langle \beta, x \rangle} \ : \
	c_{\alpha} \geqslant  0 \text{ for } \alpha \in \cA, \, c_{\beta} \in \R, \,
	f\geqslant  0 \text{ on } \R^n\Big\}
	\]
	denotes the non-negative exponential sums which may only have a negative coefficient in the
	term indexed by $\beta$
	(see \cite{chandrasekaran-shah-2016}). 
	The elements in these cones are called \emph{SAGE signomials} and 
	\emph{AGE signomials}, respectively.
	The cone
	$C_{\sage}(\mathcal{T})$ is
	a closed convex cone in $\R^{\mathcal{T}}$ (see~\cite[Proposition~2.10]{knt-2020}).
		
	Membership in this convex cone can be decided in terms of 
	relative entropy programming.
	For a finite set $\emptyset \neq \cA \subset  \R^n$, denote by
	$D:\R_{>0}^\cA \times \R_{>0}^\cA \to \R$,
	\[
	D(\nu, \gamma) \ = \ \sum_{\alpha \in \cA} \nu_\alpha \ln \left( \frac{\nu_\alpha}{\gamma_\alpha} \right)
	\]
	the \emph{relative entropy function}, which can 
	be extended to $\R_{+}^\cA \times \R_{+}^\cA \to \R \cup \set{\infty}$ via the conventions
	$0 \cdot \ln \frac{0}{y} = 0$ for $y \geqslant  0$ and $y \cdot \ln \frac{y}{0} = \infty$ for $y > 0$.
	To decide membership of a given signomial $f$ supported on $\mathcal{T}$
	in the SAGE cone, assume that $f$ is written in the form
	\[
	f = \sum_{\alpha \in \mathcal{A}} c_{\alpha} \exp(\langle \alpha, x \rangle)
	+ \sum_{\beta \in \mathcal{B}} c_{\beta} \exp(\langle \beta, x \rangle)
	\]
	with $c_{\alpha} > 0$ for $\alpha \in \mathcal{A}$
	and $c_{\beta} < 0$ for $\beta \in \mathcal{B}$.
	In this notation, the overall support set of $f$ is $\cX =\cA \cup \cB$.
  Accordingly, for disjoint sets $\emptyset\ne\mathcal{A}\subset \R^n$ and $\cB\subset \R^n$,
  it is convenient to denote by 
	\begin{equation}
	\label{eq:sage-a-b}
	C_{\mathrm{SAGE}}(\cA,\cB) := \sum_{\beta \in \cB}
	C_{\mathrm{AGE}}(\cA\cup\cB\setminus\{\beta\} ,\beta)
	\end{equation}
	the \textit{signed SAGE cone}, which allows negative coefficients only
	in a certain subset $\cB$ of the support $\cA\cup\cB$. This is a common notation in
	optimization viewpoints \cite{dhnw-2020,didW-2019,iliman-de-wolff-2016-siopt,
	mcw-2018,mcw-2019}.
	
	\begin{prop}[\cite{mcw-2018}]\label{prop:relentr_SAGE}
		A signomial 		
		$f$ belongs to $C_{\mathrm{SAGE}}(\cA , \cB)$ if and only if
		for every $\beta \in \mathcal{B}$ there exist 
		$c^{(\beta)} \in \R_+^\cA$ and
		$\nu^{(\beta)} \in \R_+^{\mathcal{A}}$ such that
		\[
		 \begin{array}{rcll}
		\sum\limits_{\alpha \in \cA}
		\nu_{\alpha}^{(\beta)} \alpha 
		& = & (\sum\limits_{\alpha \in \cA} \nu_{\alpha}^{(\beta)})\beta & \text{ for }\beta\in\cB, \nonumber \\
		D(\nu^{(\beta)}, e \cdot c^{(\beta)}) & \leqslant & c_{\beta} & \text{ for }\beta\in\cB, \label{eq:entropy1} \\
		\sum\limits_{\beta \in \mathcal{B}} c_{\alpha}^{(\beta)} & \leqslant & c_{\alpha}
		& \text{ for } \alpha \in \mathcal{A}. \nonumber
		\end{array}
		\]
		\end{prop}
	
	Note that this proposition reflects the statement of Murray, Chandrasekaran and Wierman \cite{mcw-2018}
	that every SAGE signomial can be decomposed into AGE signomials in such a way
	that every term 
	with a negative coefficient only appears in a single AGE signomial.

	\subsection*{Optimizing over the SAGE cone}

	Since the SAGE cone is contained in the cone of non-negative signomials, relaxing to the SAGE cone gives an approximation of the global infimum 
	$f^*$ of a signomial $f$ supported on $\mathcal{T}$: 
	\[
	f^{\mathrm{SAGE}} := \sup \{ \lambda \in \R \, : \ f - \lambda \in C_{\mathrm{SAGE}}(\mathcal{T}) \}
	\]  
	satisfying $f^{\mathrm{SAGE}}\leqslant f^*$.

	\subsection*{Constrained versions}	
	
	While many aspects of this article are devoted to the unconstrained situation,
	we briefly collect the extension of SAGE certificates to the constrained situation.
	Let $K$ be a convex and closed subset of $\R^n$.
	For a convex set $K \subset \R^n$ and a non-empty finite set $\mathcal{T} \subset \R^n$,
	the $K$-SAGE cone
	$C_{K}(\cX)$ is defined (see \cite{mcw-2019}) as
	\[
	C_{K}(\cX) := \sum_{\beta \in \cX}
	C_{K}(\cX \setminus \{\beta\},\beta),
	\]
	where for $\cA:=\cX\setminus\{\beta\}$,
	
	\[
	C_K(\cA,\beta) := \Big\{ f = 
	\sum\limits_{\alpha\in\cA} c_\alpha e^{\langle \alpha, x \rangle}
	+ c_\beta e^{\langle \beta, x\rangle} \ : \
	c_{\alpha} \geqslant  0 \text{ for } \alpha \in \cA, \, c_{\beta} \in \R, \,
	f \geqslant  0 \text{ on } K \Big\}.
	\]
 Moreover, \eqref{eq:sage-a-b} can be generalized by defining,
	for disjoint sets $\emptyset\ne\mathcal{A}\subset \R^n$ and $\cB\subset \R^n$, 
	the \textit{signed $K$-SAGE cone}
	\[
	C_{K}(\cA,\cB) := \sum_{\beta \in \cB}
	C_{K}(\cA ,\beta).
	\]
	This is the set of $K$-SAGE signomials, where negative coefficients are only possible
	in a certain subset $\cB$ of the support $\cA\cup\cB$. The following decomposition
	result holds.
	
	\begin{thm}[\cite{mcw-2019}, Corollary $5$] \label{thm:decomposition_theorem}
		If $f \in C_K(\cA,\cB)$ with $c_{\alpha} > 0$ for all $\alpha \in \cA$
		and $c_{\beta} < 0$ for all $\beta \in \cB \neq \emptyset$, then there
		exist $K$-AGE signomials $f_{\beta} \in C_K(\cA,  \beta)$ for 
		$\beta \in \cB$ such that $f = \sum_{\beta \in \cB} f_{\beta}$.
	\end{thm}
	
	For the constrained approach, a similar result to Proposition~\ref{prop:relentr_SAGE} is known.
	
	\begin{prop}[\cite{mcw-2019}]\label{prop:X_SAGE_relative_entropy}
		$f \in C_K(\cA, \cB)$ if and only if
		for every $\beta \in \mathcal{B}$ there exist 
		$c^{(\beta)} \in \R_+^{\mathcal{A}}$ and
		$\nu^{(\beta)} \in \R_+^{\mathcal{A}}$ such that
		\[
		  \begin{array}{rcll}
		D(\nu^{(\beta)}, e \cdot c^{(\beta)}) + 
		\sup\limits_{x \in K} \langle  -\sum\limits_{\alpha \in \cA}
		\nu_{\alpha}^{(\beta)} (\alpha -\beta), x \rangle & \leqslant & c_{\beta} & \text{ for }\beta\in\cB, \\
		\sum\limits_{\beta \in \mathcal{B}} c_{\alpha}^{(\beta)} & \leqslant & c_{\alpha} &
		\text{ for } \alpha \in \mathcal{A}.
		\end{array}
		\]	
	\end{prop}

\section{Orbit decompositions of symmetric exponential sums\label{se:symmetrization}}
	
	In this section, we provide a structural result on the decomposition of symmetric
	SAGE exponentials as
	sums of orbits of (non-symmetric) AGE exponentials.

	Let $G$ be a finite group acting linearly on $\R^n$ on the left, namely we have a group homomorphism
	\[
	\fonc{\varphi}{G}{\mathrm{GL}_n(\R)}{\sigma}{ \varphi(\sigma)}.
	\]
	For $\sigma\in G$ and $x\in \R^n$, we denote by
	$\sigma \cdot x$ the image of $x$ through $\varphi(\sigma)$.
	In order to get a left action on the set of functions defined on $\R^n$, we need to take
	\begin{equation}
	\label{eq:induced-action1}
	(\sigma * f)(x)=f(\sigma^{-1} \cdot x)= f(\varphi(\sigma^{-1})(x)).
	\end{equation}
	For a signomial 
	$f(x) = \sum_{\alpha} c_{\alpha} \exp( \langle \alpha, x \rangle)$, we see an exponent vector $\alpha$ as an element of the dual space. 
	Then, the dual action of $G$ on the exponent vectors is given by
	\[
	\sigma \perp \alpha := \varphi(\sigma^{-1})^\# (\alpha),
	\]
	where $A^\#$ denotes the adjoint operator of $A$. 
	Note that this is a left action as well. 
	Therefore, even if the exponents and the variables lie in isomorphic spaces, the actions of $G$ on these spaces are different and dual to each other, and satisfy
	\[
	\langle \alpha, \sigma \cdot x \rangle = \langle \alpha, \varphi(\sigma)(x) \rangle = \langle \varphi(\sigma)^\# (\alpha) , x \rangle = \langle \sigma^{-1} \perp \alpha , x \rangle
	\]
	and furthermore, for a signomial $f$,
	\begin{equation}\label{eq:sig-act}
	(\sigma * f) (x) = f(\sigma^{-1} \cdot x) = \sum_{\alpha} c_{\alpha} \exp( \langle \alpha, \sigma^{-1} \cdot x \rangle) = \sum_{\alpha} c_{\alpha} \exp( \langle \sigma \perp \alpha, x \rangle).	
	\end{equation}
From now on, in order to keep notation as light as possible, with a slight abuse of notation, we write $\sigma (x) = \sigma \cdot x$ for the action on the variables, $\sigma f = \sigma *f$ for the action on functions, and $\sigma (\alpha) = \sigma \perp \alpha$ for the dual action. 
Even if the actions are different, the context should clarify the correspondence.  	
	
	For a set $\mathcal{S} \subset \R^n$ of exponent vectors,
	the \emph{orbit of} $\mathcal{S}$ under $G$ is 
	$$G\cdot \mathcal{S} = \{\sigma(s) : s \in \mathcal{S}, \, \sigma \in G\}.$$
	We call a subset 	$\hat{\mathcal{S}} \subset  \mathcal{S}$ \emph{a set of orbit representatives
		for $\mathcal{S}$}
	if $\hat{\mathcal{S}}$ is an inclusion-minimal set with
	$(G\cdot \hat{\mathcal{S}}) = \mathcal{S}$. 
	Moreover,
	let $\Stab \beta := \{\sigma \in G \ : \ \sigma(\beta) = \beta\}$ denote
	the \emph{stabilizer} of an exponent vector $\beta$.
	
	In the following statements, we consider $G$-invariant signomials $f$. 
	It is convenient to write $f$ here in the form
	\begin{equation}
	\label{eq:symm-form1}
	f = \sum_{\alpha \in \mathcal{A}} c_{\alpha} \exp(\langle \alpha, x \rangle) 
	+ \sum_{\beta \in \mathcal{B}} c_{\beta} \exp(\langle \beta, x \rangle)
	\end{equation}
	with $c_{\alpha} > 0$ for $\alpha \in \mathcal{A}$
	and $c_{\beta} < 0$ for $\beta \in \mathcal{B}$. As already mentioned in connection with the
	definition of the signed SAGE cone in~\eqref{eq:sage-a-b}, the overall
	support set of $f$ is $\cA \cup \cB$. 
	
        The following theorem shows a natural decomposition of a $G$-invariant
        signomial $f$ by means of a set of orbit representatives 
        $\hat{\mathcal{B}}$
        of $\mathcal{B}$. For every representative 
        $\hat{\beta} \in \hat{\mathcal{B}}$, it is not necessary to take into
        account the action of all permutations $\sigma \in G$, but it
        suffices to consider the possibly smaller set $G / \Stab(\hat{\beta})$.

	\begin{thm}\label{th:symmetric-decomp}
		Let $K \subset \R^n$ be convex and $G$-invariant,
		let $f$ be a $G$-invariant signomial of the form~\eqref{eq:symm-form1}
		and $\hat{\cB}$ be a set of orbit representatives for $\cB$.
		Then $f \in C_K(\cA, \cB)$ 
		if and only if for every $\hat{\beta} \in \hat{\mathcal{B}}$,
		there exists a $K$-AGE signomial 
		$h_{\hat{\beta}} \in C_K(\cA,\hat{\beta})$ 
		such that
		\begin{equation}\label{eq:symm_decomp}
				f = \sum_{\hat{\beta}\in \hat{\cB}} \sum_{\rho \in G/\Stab(\hbeta)} \rho h_{\hbeta}.	
		\end{equation}
		The functions $h_{\hat{\beta}}$ can be chosen to be invariant under the action of 
		$\Stab(\hat{\beta})$.
	\end{thm}
	
	Here, $\rho \in G/\Stab(\hbeta)$ shortly denotes
	that $\rho$ runs over a set of representatives of the left quotient space
	$G/\Stab(\hbeta)$, which is defined through the left cosets 
	$\{ \sigma \Stab(\hbeta) \, : \, \sigma \in G\}$. We will also use the right quotient
	space, denoted by $\Stab(\hbeta) \backslash G$, 
	further below.
	To illustrate the theorem, we give an example.
	
	\begin{example}
     Let $K:=\R^3$ and $G:=\mathcal{S}_3$ be the symmetric group on three variables.
     We consider the signomial
     \[ f \ = \ e^{6x_1} + e^{6x_2} + e^{6 x_3} + e^{x_1+x_2+x_3} 
       -  \delta(e^{x_1+2x_2+2x_3}
       + e^{2 x_1+x_2+2x_3}
       + e^{2 x_1+2x_2+x_3})
     \] 
     with some constant $\delta \in \R$.
     Following the notation of Theorem~\ref{th:symmetric-decomp}, let 
     $\mathcal{B} = \{(1,2,2)^T,$ $(2,1,2)^T,(2,2,1)^T\}$,
 and choose  $\hbeta = (1,2,2)^T$ as representative
    of the single $\mathcal{S}_3$-orbit in $\mathcal{B}$. This gives 
    $\Stab(\hbeta) = \{ \mathrm{id}, (2,3) \}$. 
    The left cosets are 
    $\{\mathrm{id}, (2,3) \}$, $\{(1,2,3),(1,2)\}$ as well as $\{(1,3,2),(1,3)\}$,
    so that we can choose representatives to write 
    $\mathcal{S}_3/\Stab(\hbeta) = \{\mathrm{id},(1,2,3),(1,3,2)\}$.
    By Theorem~\ref{th:symmetric-decomp}, the signomial $f$ is SAGE if and only if there exist $a,b,c,d \geqslant 0$ such that 
    \[
      h_{\hbeta} = a e^{x_1+x_2+x_3}
       + b e^{6x_1}+c e^{6x_2} + d e^{6 x_3}
       - \delta e^{x_1+2x_2+2x_3}
    \]
       is an AGE signomial, invariant under the action of $\Stab(\hbeta)$ and satisfies 
    condition~\eqref{eq:symm_decomp}, that is,
    \begin{align*}
      f = \ & 3a e^{x_1+x_2+x_3} + (b+c+d) e^{6 x_1}
          + (b+c+d) e^{6 x_2} + (b+c+d) e^{6 x_3} \\
         & - \delta(e^{x_1+2x_2+2x_3}
       + e^{2 x_1+x_2+2x_3}
       + e^{2 x_1+2x_2+x_3}).
     \end{align*}
     This implies
     $3a=1$, $c=d$ and $b+c+d=1$. With this decomposition, it can be shown that the maximal choice for $\delta$ is 
    $\delta = \sqrt[3]{\frac{9}{4}}$, which occurs when $a=\frac 13$, $b= \frac 16$ and $c=d=\frac 5{12}$. 
 %   $h_{\hbeta}$ is an AGE signomial and thus non-negative,
  %  which can be seen by verifying that it satisfies the AGE special case of
   % the relative entropy condition 
   % in Proposition~\ref{prop:relentr_SAGE}.
    % Internal note: The zero of $h_{\beta}$ is at the diagonal
    % point (x,x,x) with x := -1/3 * ln(3/2)
	\end{example}

\iffalse
		\begin{example}
     For $K=\R^n$, the signomial
     \[ f \ = \ e^{6x_1} + e^{6x_2} + e^{6 x_3} + e^{x_1+x_2+x_3} 
       -  \delta(e^{x_1+2x_2+2x_3}
       - e^{2 x_1+x_2+2x_3}
       - e^{2 x_1+2x_2+x_3})
     \] 
     is a SAGE signomial if $\delta \geqslant \sqrt[3]{\frac{9}{4}}$. Namely, with
     respect to the symmetric group in three variables and 
     $\mathcal{B} = \{(1,2,2)^T,(2,1,2)^T,(2,2,1)^T\}$,
    the decomposition provided by Theorem~\ref{th:symmetric-decomp}
    looks as follows. Choosing $\hbeta = \{(1,2,2)^T\}$ as representative
    of the single orbit in $\mathcal{B}$ gives 
    $\Stab(\hbeta) = \{ \mathrm{id}, (2,3) \}$ so that 
    $|G/\Stab(\hbeta)| = \{\mathrm{id},(1,2,3),(1,3,2)\}$.
    The signomial
    \[
      h_{\hbeta} = \frac{1}{3} e^{x_1+x_2+x_3}
       + \frac{1}{6} e^{6x_1}+\frac{5}{12} e^{6x_2} + \frac{5}{12} e^{6 x_3}
       - \delta e^{x_1+2x_2+2x_3}
    \]
    is invariant under the action of $\Stab(\hbeta)$ and satisfies 
    condition~\eqref{eq:symm_decomp}. For 
    $\delta \geqslant \sqrt[3]{\frac{9}{4}}$,
    $h_{\hbeta}$ is an AGE signomial and thus non-negative,
    which can be seen by verifying that it satisfies the AGE special case of
    the relative entropy condition 
    in Proposition~\ref{prop:relentr_SAGE}.
    % Internal note: The zero of $h_{\beta}$ is at the diagonal
    % point (x,x,x) with x := -1/3 * ln(3/2)
	\end{example}
\fi

	\begin{proof}
	Since it is clear that a signomial $f$ of the form \eqref{eq:symm_decomp} is non-negative, we only have to show the converse direction.
		Let 
		$f \in C_K(\cA, \cB)$. By Theorem \ref{thm:decomposition_theorem}, 	     
		there exist $K$-AGE signomials 
		$f_{\beta}\in C_K(\cA,\beta)$ for $\beta \in \cB$,
		such that
		$f=\sum_{\beta \in \cB} f_\beta.$
		The $G$-invariance of $f$ gives
		\begin{equation}
		\label{eq:decomp-proof1}
		f= \frac{1}{|G|} \sum_{\sigma\in G} \sigma f = \frac{1}{|G|} \sum_{\sigma\in G} \sum_{\beta \in \cB} \sigma f_\beta.
		\end{equation}
		The idea is to group in this sum all the $\sigma f_\beta$ that have the same ``possibly negative'' term.
According to \eqref{eq:sig-act}, the possibly negative term of $\sigma f_\beta$ is given by $\sigma (\beta) $. 
		For any $\beta \in \cB$, the signomial
		\[
		h_\beta = \frac{1}{|G|} \sum_{ \sigma \in G} \sigma f_{\sigma^{-1} (\beta) }
		\] 
		is a sum of $K$-AGE signomials in $C_K(\cA, \beta)$,
		hence it is contained in $C_K(\cA , \beta)$ as well.
		Moreover, \eqref{eq:decomp-proof1}
		can be expressed as
		\[
		f = \frac{1}{|G|} \sum_{\sigma \in G} \sum_{\beta \in \cB}  \sigma f_\beta
		= \frac{1}{|G|} \sum_{\sigma \in G} \sum_{\gamma \in \cB}  \sigma f_{\sigma^{-1} (\gamma)}
		=  \sum_{\gamma \in \cB}  h_{\gamma}.
		\] 
		
		Let $\beta \in \cB$ and $\hbeta \in \hat{\cB}$ be the representative of its orbit in $\hat{\cB}$. If $\sigma , \tau \in G$ are such that $\sigma (\hbeta) = \tau(\hbeta) = \beta$, then 
		$\tau^{-1} \sigma \in \Stab(\hbeta)$ and $\tau = \sigma$ in $G / \Stab(\hbeta)$. 		
		Hence,
		\begin{align}
		\label{eq:decomp-proof-split}
		f &= \sum_{\hbeta \in \hat{\cB} } \sum_{\rho \in G/\Stab \hbeta} h_{\rho (\hbeta)}.
		\end{align}
		Now observe that 
		$
		h_{\rho (\beta)} = \rho h_{\beta} \text{ for every } \beta \in \cB$ and $\rho \in G,
		$
		because
		\begin{align}
		\label{eq:decomp-proof2}
		|G|\rho h_{\beta} = \sum_{ \sigma \in G} \rho \sigma f_{\sigma^{-1} (\beta) }
		= \sum_{ \tau \in G} \tau f_{\tau^{-1} \rho (\beta) } =|G| h_{\rho (\beta)}.
		\end{align}
		Substituting~\eqref{eq:decomp-proof2} into~\eqref{eq:decomp-proof-split} gives
		$
		f 
		= \sum_{\hbeta \in \hat{\cB} } \sum_{\rho \in G/\Stab \hbeta} \rho h_{\hbeta}
		$ as desired.
		Moreover, the $\Stab(\hat{\beta})$-invariance of 
		$h_{\hat{\beta}}$ 
		for $\hbeta \in \hat{\mathcal{B}}$ follows from~\eqref{eq:decomp-proof2}.
	\end{proof}
		
\begin{rem}
Note that the previous results extend naturally to compact/reductive groups, since they mainly rely on the existence of a Reynolds operator. 
For the sake of simplicity, we presented them for finite groups, where the Reynolds operator corresponds to a finite average over the group.
\end{rem}

\section{Symmetry reduction in relative entropy programming}\label{se:sym-re}

Building upon the previous decomposition theorem, we provide a symmetry-adapted
	relative entropy formulation for containment in the SAGE cone.

	\begin{thm}	\label{co:entropy-symm1}
		Let  $\hat{\cB}$ be a set of orbit representatives for $\cB$.
		A $G$-invariant signomial $f$ of the form~\eqref{eq:symm-form1}
		is contained in $C_{\mathrm{SAGE}}(\cA,\cB)$ if and only if 
		for every $\hbeta \in \hat{\cB}$ there exist
		$c^{(\hbeta)} \in \R_+^{{\mathcal{A}}}$ and $\nu^{(\hbeta)} \in \R_+^{{\mathcal{A}}}$
		, invariant under the action of $\Stab(\hbeta)$,
		such that
		\begin{align}
		\sum_{\alpha \in {\cA}} 
		\nu_{\alpha}^{(  {\hat{\beta}}  )} (\alpha -\hbeta)
		& \ = \ 0 \quad \text{ for every } \hbeta\in\hat{\cB}, \label{eq:convexhull-symm1}\\
		D(\nu^{(\hbeta)}, e \cdot c^{(\hbeta)}) & \ \leqslant \ c_{\hbeta} \quad \text{for every } \hbeta\in\hat{\cB}, \label{eq:entropy-symm1} \\
		\sum_{\hbeta \in \hat{\mathcal{B}}}
		\sum_{\sigma \in \Stab{(\hbeta)} \backslash G   } c_{\sigma(\alpha)}^{(\hbeta)} & \ \leqslant \  c_{\alpha}
		\quad \text{for every } \alpha \in {\mathcal{A}}.
		\label{eq:coeff-symm1}
		\end{align}
	\end{thm}
	
	\begin{rem}\label{rem:entropy-symm1}
		The right coset condition~\eqref{eq:coeff-symm1} can equivalently be expressed
		in terms of the left cosets,
		\[
		\sum_{\hbeta \in \hat{\mathcal{B}}}
		\sum_{\sigma \in G / \Stab{\hbeta}} c_{\sigma^{-1}(\alpha)}^{(\hbeta)} 
		\ \leqslant \  c_{\alpha}
		\quad \text{ for every } \alpha \in {\mathcal{A}}.
		\] Namely, if  $\beta \in \cB$, $\hbeta \in \hat{\cB}$ and $\sigma , \tau \in G$ are such that 
		$\sigma^{-1}(\hbeta) = \tau^{-1}(\hbeta) = \beta$, then 
		$\tau \sigma^{-1} \in \Stab(\hbeta)$ and $\tau = \sigma$ in the right quotient
		space $\Stab(\hbeta)\backslash G$. 		
	\end{rem}
	
	\begin{proof}[Proof of Theorem \ref{co:entropy-symm1}]
		If $f$ is $G$-symmetric, then, by
		Theorem~\ref{th:symmetric-decomp}, there exist $\Stab(\hbeta)$-invariant AGE signomials
		$h_{\hat{\beta}} \in C_{\mathrm{SAGE}}(\cA , \hat{\beta})$ for every
		$\hat{\beta} \in \hat{\mathcal{B}}$ 
		such that
		\[
		f = \sum_{\hat{\beta}\in \hat{\cB}} \sum_{\rho \in G/\Stab(\hbeta)} \rho h_{\hbeta}.
		\]
		Writing $h_{\hat{\beta}}$ in the form
		\[
		h_{\hat{\beta}} = 
		\sum_{\alpha \in \cA} c_{\alpha}^{(\hat{\beta})} \exp(\langle \alpha, x \rangle)
		+ c_{\hat{\beta}} \exp(\langle \hat{\beta}, x \rangle)
		\]
		with coefficients $c_{\alpha}^{(\hbeta)}$
		and $c_{\hbeta}$ 
		for $\alpha \in \cA$ and $\hbeta \in \hat{\cB}$,
		the two conditions~\eqref{eq:convexhull-symm1} 
		and~\eqref{eq:entropy-symm1} follow from the property 
		$h_{\hbeta} \in C_{\mathrm{SAGE}}(\cA, \hat{\beta})$.
		For~\eqref{eq:coeff-symm1}, we observe that for $\alpha \in \cA$,
		the coefficient 
		of $\exp( \langle \alpha, x \rangle)$ in $\rho h_{\hbeta}$ is
		$c^{(\hbeta)}_{\rho^{-1}(\alpha)}$.
		We obtain inequality \eqref{eq:coeff-symm1}, even with equality, by setting $\sigma := \rho^{-1}$
		and summing over 
		$\hbeta \in \hat{\cB}$ and over $\sigma \in \Stab(\hbeta) \backslash G$, following Remark~\ref{rem:entropy-symm1}. Moreover, the
		$\Stab(\hbeta)$-invariance of $h_{\hbeta}$ implies the $\Stab(\hbeta)$-invariance 
		of  $c^{(\hbeta)}$. In order to make $\nu^{(\hbeta)}$  invariant under $\Stab(\hbeta)$, we can replace it with \[\mu_\alpha^{(\hbeta)}= \frac{1}{|\Stab(\hbeta)|} \sum_{\sigma \in \Stab(\hbeta)} \nu_{\sigma(\alpha)}^{(\hbeta)}.\] Obviously, this has no influence on ~\eqref{eq:coeff-symm1}. For~\eqref{eq:convexhull-symm1}, we have \begin{align*}
			{|\Stab(\hbeta)|} \sum_{\alpha \in {\cA}} \mu_{\alpha}^{(  {\hat{\beta}}  )} (\alpha -\hbeta) &= 	\sum_{\alpha \in {\cA}} \sum_{\sigma \in \Stab(\hbeta)}\nu_{\sigma(\alpha)}^{(  {\hat{\beta}}  )} (\alpha -\hbeta) \\ &=  \sum_{\sigma \in \Stab(\hbeta)} \sigma^{-1} \sum_{\alpha \in \cA} \nu_{\sigma(\alpha)}^{(\hbeta)} (\sigma(\alpha) - \sigma(\hbeta))\\ &= 		\sum_{\sigma \in \Stab(\hbeta)} \sigma^{-1} \sum_{\alpha \in \cA} \nu_{\alpha}^{(\hbeta)} (\alpha - \hbeta)   = 0.\end{align*} 
		Finally, for~\eqref{eq:entropy-symm1}, using $c_\alpha^{(\hbeta)} = c_{\sigma(\alpha)}^{(\hbeta)}$ for $\sigma \in \Stab(\hbeta)$ and applying Jensen's inequality on the convex function
		$x \mapsto x \ln x$ gives, for all $\alpha \in \cA$,
		\begin{align*} \mu_{\alpha}^{(\hbeta)} \ln \frac {\mu_{\alpha}^{(\hbeta)}}{c_{\alpha}^{(\hbeta)}} &= \left(\frac{1}{|\Stab(\hbeta)|} \sum_{\sigma \in \Stab(\hbeta)} \nu_{\sigma(\alpha)}^{(\hbeta)} \right) \ln \frac {\frac{1}{|\Stab(\hbeta)|} \sum_{\sigma \in \Stab(\hbeta)} \nu_{\sigma(\alpha)}^{(\hbeta)}}{c_{\alpha}^{(\hbeta)}}
		\\&= c_\alpha^{(\hbeta)} \left(\frac{\sum_{\sigma \in \Stab(\hbeta)}\nu_{\sigma(\alpha)}^{(\hbeta)}/c_{\sigma(\alpha)}^{(\hbeta)}}{|\Stab(\hbeta)|} \ln \frac{\sum_{\sigma \in \Stab(\hbeta)}\nu_{\sigma(\alpha)}^{(\hbeta)}/c_{\sigma(\alpha)}^{(\hbeta)}}{|\Stab(\hbeta)|}\right)
		\\
		&\leqslant c_\alpha^{(\hbeta)}\left( 
		\frac{1}{|\Stab(\hbeta)|}
		\sum_{\sigma \in \Stab(\hbeta)} \frac{\nu_{\sigma(\alpha)}^{(\hbeta)}}
		{c_{\sigma(\alpha)}^{(\hbeta)}} \ln \frac{\nu_{\sigma(\alpha)}^{(\hbeta)}}
		{c_{\sigma(\alpha)}^{(\hbeta)}}
		\right).
		\end{align*}
		Using again the $\Stab(\hbeta)$-invariance of  $c^{(\hbeta)}$ and the precondition then 
		yields
		\begin{align*}
		\sum_{\alpha \in \cA} \mu_{\alpha}^{(\hbeta)} \ln \frac {\mu_{\alpha}^{(\hbeta)}}{ec_{\alpha}^{(\hbeta)}} 
		\leqslant \frac{1}{|\Stab(\hbeta)|} \sum_{\sigma \in \Stab(\hbeta)} \sum_{\alpha \in \cA} \nu_{\sigma(\alpha)}^{(\hbeta)} \ln \frac{\nu_{\sigma(\alpha)}^{(\hbeta)}}{ec_{\sigma(\alpha)}^{(\hbeta)}}
		\leqslant  \frac{1}{|\Stab(\hbeta)|} \sum_{\sigma \in \Stab(\hbeta)} c_{\hbeta} = c_{\hbeta}.
		\end{align*}
		
		Conversely, assume that $c^{(\hbeta)}$ and $\nu^{(\hbeta)}	$, invariant under the action of $\Stab(\hbeta)$,
		satisfy~\eqref{eq:convexhull-symm1}--\eqref{eq:coeff-symm1}. 
		Let $\beta \in \cB$ and $\hbeta \in \hat{\cB}$ be the representative of its orbit in $\hat{\cB}$. If $\sigma , \tau \in G$ are such that $\sigma (\beta) = \tau(\beta) = \hbeta$, then $\tau \sigma^{-1} \in \Stab(\hbeta)$ and $\tau = \sigma$ in $\Stab(\hbeta)\backslash G$. Since $c^{(\hbeta)}$ and $\nu^{(\hbeta)}$ are invariant under $\Stab(\hbeta)$, we have 
		\[
		c^{(\hbeta)}_{\tau(\alpha)} = c^{(\hbeta)}_{\sigma(\alpha)}, \quad
		\nu^{(\hbeta)}_{\tau(\alpha)} = \nu^{(\hbeta)}_{\sigma(\alpha)}
		\quad \text{ for } \alpha \in \cA.
		\] Thus we can define \[
		c^{(\beta)}_{\alpha} = c^{(\hbeta)}_{\sigma(\alpha)}, \quad
		\nu^{(\beta)}_\alpha = \nu^{(\hbeta)}_{\sigma(\alpha)}
		\quad  \text{ for } \alpha \in \cA, \] 
		which is independent of $\sigma$ such that $\sigma(\beta)=\hbeta$. 
		As a consequence, if $\tau \in \Stab(\hbeta)\backslash G$, then $c_\alpha^{(\tau^{-1}(\hbeta))} = c_{\tau(\alpha)}^{(\hbeta)}$ is well defined.
		
		To see that the first conditions of
		Proposition~\ref{prop:relentr_SAGE}
		are satisfied, let $\beta \in \cB$ and $\sigma\in G$ such that $\sigma (\beta) = \hbeta$. Then 
		\begin{align*}
		\sum_{\alpha \in \cA} \nu_\alpha^{(\beta)} (\alpha - \beta) & = \sum_{\alpha \in \cA} \nu_{\sigma(\alpha)}^{(\hbeta)} (\alpha - \sigma^{-1}(\hbeta)) \\&= \sigma^{-1} \sum_{\alpha \in \cA} \nu_{\sigma(\alpha)}^{(\hbeta)} (\sigma(\alpha) - \hbeta) = \sigma^{-1} \sum_{\alpha \in \cA} \nu_\alpha^{(\hbeta)} (\alpha - \hbeta) = 0 \\
		\text{ and } \quad
		D(\nu^{(\beta)},e c^{(\beta)}) & = D(\nu^{(\hbeta)},e c^{(\hbeta)}) \leqslant c_{\hbeta} = c_\beta.
		\end{align*}
		For the third condition of Proposition~\ref{prop:relentr_SAGE},
		we obtain \[\sum_{\beta \in \cB} c_\alpha^{(\beta)} = \sum_{\hbeta \in \hat{\cB}} \sum_{\tau \in \Stab(\hbeta)\backslash G} c_\alpha^{(\tau^{-1}(\hbeta))} = \sum_{\hbeta \in \hat{\cB}} \sum_{\tau \in \Stab(\hbeta)\backslash G} c_{\tau(\alpha)}^{(\hbeta)} \leqslant c_\alpha,\]         
		which altogether shows that $f \in C_{\mathrm{SAGE}}(\cA ,\cB)$.
	
	\end{proof}
	
	The following consequence of 
	Theorem~\ref{co:entropy-symm1} further reduces
	the number of variables in the relative entropy program, since a certain number 
	of $c_{\alpha}^{(\hbeta)}$ and $\nu_\alpha^{(\hbeta)}$ are actually equal, and we 
	can take each $c^{(\hbeta)}, \nu^{(\hbeta)}$ in the ground set $\R_+^{\cA / \Stab(\hbeta)}$. 
	
	\begin{cor}
		\label{co:reduce-a}
		Let  $\hat{\cA}$ and $\hat{\cB}$ be a set of orbit representatives for $\cA$ and $\cB$.
		A $G$-invariant signomial $f$ of the form~\eqref{eq:symm-form1}
		is contained in $C_{\mathrm{SAGE}}(\cA,\cB)$ if and only if 
		for every $\hbeta \in \hat{\cB}$ there exist
		$c^{(\hbeta)} \in \R_+^{\cA / \Stab(\hbeta)}$ and 
		$\nu^{(\hbeta)} \in \R_+^{\cA / \Stab(\hbeta)}$
		such that
		\begin{eqnarray}
		\sum_{\alpha \in \cA / \Stab(\hbeta)} \nu_{\alpha}^{(\hbeta)}\sum_{\alpha' \in \Stab(\hbeta) \cdot \alpha} (\alpha'-\hbeta) & = & 0 \quad \text{ for every } \hbeta \in \hat{\cB}, 
		\label{eq:reduce-a1} \\
		\sum_{\alpha \in \cA/ \Stab(\hbeta)} \left|\Stab(\hbeta)\cdot \alpha\right| \nu_\alpha^{(\hbeta)} \ln \frac{\nu_\alpha^{(\hbeta)}}{ec_\alpha^{(\hbeta)}} & \leqslant & c_{\hbeta} \quad 
		\text{for every } \hbeta \in \hat{\cB}, 
		\label{eq:reduce-a2} \\
		\sum_{\hbeta \in \hat{\cB}} \frac{|\Stab (\alpha)|}{|\Stab(\hbeta)|} \sum_{\gamma \in (G\cdot \alpha)/\Stab(\hbeta)}\left|\Stab(\hbeta)\cdot \gamma\right| c_\gamma^{(\hbeta)} & \leqslant & c_\alpha \quad \text{for every } \alpha \in \hat{\cA}.
		\label{eq:reduce-a3} 
		\end{eqnarray} 
	\end{cor}
	
	\begin{proof} 
		For~\eqref{eq:reduce-a1} and~\eqref{eq:reduce-a2}, equivalence to their versions
		in Theorem~\ref{co:entropy-symm1} is straightforward to check. 
		For~\eqref{eq:reduce-a3}, equivalence to~\eqref{eq:coeff-symm1} follows by
		observing that for every $\alpha \in \cA$
		\begin{align*}  \sum_{\sigma  \in \Stab(\hbeta)\backslash G}c_{\sigma(\alpha)}^{(\hbeta)} &= \sum_{\sigma  \in \Stab(\hbeta)\backslash G}\frac{1}{|\Stab(\hbeta)|} \sum_{\tau \in \Stab(\hbeta)} c_{\tau(\sigma(\alpha))}^{(\hbeta)}
		= \frac{1}{|\Stab(\hbeta)|} \sum_{\rho \in G} c_{\rho(\alpha)}^{(\hbeta)}\\
		&= \frac{|\Stab(\alpha)|}{|\Stab(\hbeta)|} \sum_{\gamma \in G \cdot \alpha} c_{\gamma}^{(\hbeta)}
		= \frac{|\Stab (\alpha)|}{|\Stab(\hbeta)|} \sum_{\gamma \in (G\cdot \alpha)/\Stab(\hbeta)}\left|\Stab(\hbeta) \cdot \gamma\right| c_\gamma^{(\hbeta)},
		\end{align*} 
		and the last expression only depends on the orbit $G \cdot \alpha$ rather than on $\alpha$
		itself.
	\end{proof}
	
		\begin{rem}\label{rem:G.alpha}
		Note that we cannot simply assume $c_\alpha^{(\beta)}=c_{\alpha'}^{(\beta)}$ for some $\alpha'\in G\cdot\alpha$ and, similarly, we cannot simply assume $\nu_\alpha^{(\beta)}=\nu_{\alpha'}^{(\beta)}$ for some $\alpha'\in G\cdot\alpha$, for instance due to \eqref{eq:entropy1}. Namely, if an element $\beta$ lies in $\conv\cA$ with barycentric coordinates $\lambda$, say $\beta=\sum_{\alpha\in\cA}\lambda_\alpha\alpha$, 
		then for any $\sigma\in G$, we have
		\begin{align*}
		\sigma(\beta)=\sigma\left(\sum_{\alpha\in\cA}\lambda_\alpha\alpha\right)=\sum_{\alpha\in\cA}\sigma(\lambda_\alpha\alpha)=\sum_{\alpha\in\cA}\lambda_{\alpha}\sigma(\alpha)
		\end{align*}  
		rather than $\sigma(\beta)=\sigma(\sum_{\alpha\in\cA}\lambda_\alpha\alpha)
		=\sum_{\alpha\in\cA}\lambda_{\sigma(\alpha)}\sigma(\alpha)$.
		Of course, this caveat does not occur whenever there is a single inner term.
	\end{rem}
		
		For symmetric constraint sets $K$, a constrained version of 
		Theorem~\ref{co:entropy-symm1} (and similarly, of Corollary~\ref{co:reduce-a})
		can be given as well. The proof is similar.
		
		\begin{cor}
			\label{co:entropy-symm2}
			Let $K \subset \R^n$ be convex and $G$-invariant. 
			A $G$-invariant signomial $f$ of the form~\eqref{eq:symm-form1}
			is contained in $C_K(\cA,\cB)$ if and only {if} for every 
			$\hbeta \in \hat{\mathcal{B}}$ there exist 
			$c^{(\hbeta)} \in \R_+^{{\mathcal{A}}}$ and
			$\nu^{(\hbeta)} \in \R_+^{{\mathcal{A}}}$ such that
			\[
			\begin{array}{rcll}
			D(\nu^{(\hbeta)}, e \cdot c^{(\hbeta)}) + 
			\sup_{x \in K}  \langle \big(-\sum\limits_{\alpha \in {\cA}} 
			\nu_{\alpha}^{(\hbeta)} (\alpha -\hbeta) \big), x \rangle & \leqslant & c_{\hbeta}
			& \text{ for every } \hbeta \in \hat{\cB}, \\
			\sum\limits_{\hbeta \in \hat{\mathcal{B}}}
			\sum\limits_{\sigma \in \Stab \hat{\beta} \backslash G  } c_{\sigma(\alpha)}^{(\hbeta)} & \leqslant & c_{\alpha}
			& \text{ for every } \alpha \in \mathcal{A}.
			\end{array}
			\]
		\end{cor}
To close this section we discuss the resulting complexity reduction:

Note that the initial relative entropy formulation which does not take the symmetry into consideration will involve  $2|\cB||\cA|$ variables. Furthermore, since every vector equality in \eqref{eq:reduce-a1} brings $n$ scalar equalities, it will consist of  $|\cB|n+|\cB|+|\cA|$ (in)equalities. 

In contrast, let us analyze the number of variables and constraints involved in the relative entropy program in Corollary~\ref{co:reduce-a}.
Observe that $\cA /\Stab(\hat{\beta})$ is the disjoint union of the $G\cdot \hat{\alpha} / \Stab(\hat{\beta})$ where $\hat{\alpha}$ runs through $\hat{\cA}$.  It follows that for every pair $\hat{\beta} \in \hat{\cB}, \hat{\alpha}\in \hat{\cA}$, we have exactly $2| (G\cdot \hat{\alpha}) / \Stab(\hat{\beta})|$ variables $c_\gamma^{(\hat{\beta})}$ and $\nu_\gamma^{(\hat{\beta})}$.

By definition, $|(G\cdot \hat{\alpha}) / \Stab(\hat{\beta})|$ is the number of $\Stab(\hat{\beta})$-orbits  in $G\cdot \hat{\alpha}$.  Since $G\cdot \hat{\alpha}$ is in bijection with $\Stab{\hat{\alpha}}\backslash G$ we get a bijection between $(G\cdot \hat{\alpha}) / \Stab(\hat{\beta})$ and the set of double cosets $\Stab (\hat{\alpha}) \backslash G / \Stab(\hat{\beta})$.
Therefore, the number of  orbits in question equals  $| \Stab(\hat{\alpha}) \backslash G / \Stab (\hat{\beta}) |$,  satisfying, according to Burnside's Lemma (see for instance \cite[Lemma 7.24.5]{Stanley}):
\[
| \Stab(\hat{\alpha}) \backslash G / \Stab (\hat{\beta}) | = \frac{1}{|\Stab(\hat{\alpha})||\Stab(\hat{\beta})|} \sum_{\substack{\sigma \in \Stab(\hat{\alpha}) \\ \tau \in \Stab(\hat{\beta})}} |G^{\sigma, \tau}|,
\]
where $|G^{\sigma, \tau}|$ is the number of elements of $G$ fixed under the action of $(\sigma, \tau)$. 
From another point of view, this number can be interpreted in terms of representation theory as follows: It is given by the inner product of the two characters corresponding to the representations induced respectively by the trivial representations of $\Stab(\hat{\alpha})$ and $\Stab(\hat{\beta})$ on $G$ (see \cite[Exercise 7.77.a.]{Stanley} for more details).

Furthermore, \eqref{eq:reduce-a1} amounts to $|\hat{\cA}| + |\hat{\cB}|$ inequalities, together with one vector equality for every element of $\hat{\cB}$. 
We observe that for a given $\hat{\beta}$, this vector is invariant by $\Stab(\hat{\beta})$ and therefore is contained in $(\R^n)^{\Stab(\hat{\beta})}$, the subspace of $\R^n$ of points fixed by $\Stab(\hat{\beta})$. Thus, by projecting onto this subspace the number of resulting equations reduces to $\dim \left((\R^n)^{\Stab(\hat{\beta})}\right)$.
As a conclusion, we obtain:
\begin{thm}\label{thm:sizes-general}
Let  $\hat{\cA}$ and $\hat{\cB}$ be a set of orbit representatives for $\cA$ and $\cB$.
For $\hat{\alpha} \in \hat{\cA}$, $\hat{\beta} \in \hat{\cB}$, denote by ${}_{\hat{\alpha}}G_{\hat{\beta}}$ the cardinality $|\Stab(\hat{\alpha}) \backslash G / \Stab (\hat{\beta}) |$, and by $n_{\hat{\beta}}$ the dimension of the fixed subspace $(\R^n)^{\Stab(\hat{\beta})}$.
		Then, the relative entropy program in Corollary~\ref{co:reduce-a} consists of 
\[
 2 \sum_{\substack{  \hat{\alpha} \in \hat{\cA} \\ \hat{\beta} \in \hat{\cB} }} {}_{\hat{\alpha}}G_{\hat{\beta}}  \text{ variables,} 
\quad  \sum_{\hat{\beta}\in \hat{\cB}} n_{\hat{\beta} } \text{ scalar equalities, and}
\quad  |\hat{\cA}| + |\hat{\cB}|  \text{ inequalities.} 
\]
\end{thm}

The next section will make the theorem more concrete in the special case of
the symmetric group.		

\section{The case of the symmetric group}\label{se:Sn}

In this section, we focus our attention to the case of the symmetric group $\mathcal{S}_n$ acting on $\R^n$ by permutation of the coordinates: for $\sigma \in \mathcal{S}_n, x\in \R^n$,
\[
\sigma(x) = (x_{\sigma^{-1}(1)}, \ldots, x_{\sigma^{-1}(n)}). 
\]
Note that because the action is orthogonal, the dual action on the exponent vectors is the same. 
Optimization problems invariant under this action can arise in different contexts naturally, for example, in the context of graph homomorphisms (\cite{blekherman2020simple}). 
This action is very natural, and the theory of representation of the symmetric group is very well understood, and affords strong connections with combinatorics.  This connection has been successfully  used in several instances to reduce the sizes of optimization problems. In particular, it was shown in \cite[Theorem 4.7 ]{rtjl-2013}  (see also \cite[Theorem 3.21]{debus2020reflection}) that the size of a semi-definite program which certifies if a given symmetric polynomial is a sum of squares is stabilizing once the number of variables is big enough. 
Similar to these results, we show in Theorem \ref{thm:stabilization} an analogous result of stabilization in the AM/GM setup. This result mainly stems from the fact that  the cardinalities appearing in Theorem~\ref{thm:sizes-general} have a combinatorial interpretation in the context of symmetric group actions.

From now on, we use
the symbols $\lambda$, $\mu$ to denote partitions. This should
cause no confusion to the use of these symbols in other contexts
in earlier sections.

First, up to permutation, every $\alpha \in \R^n$ is of the form
\[
\alpha = (\underbrace{\alpha_1,\ldots, \alpha_1}_{\lambda_1},\underbrace{\alpha_2,\ldots, \alpha_2}_{\lambda_2},\ldots, \underbrace{\alpha_k,\ldots, \alpha_k}_{\lambda_k}),
\]
with $\lambda_1 \geqslant \lambda_2 \ldots \geqslant \lambda_k >0$  and thus its stabilizer $\Stab (\alpha)$ is, up to conjugation, of the form
\[
\mathcal{S}_{\lambda_1} \times \cdots \times \mathcal{S}_{\lambda_k},
\]
so that $| \Stab ( \alpha ) |= \lambda_1 ! \cdots \lambda_k !          $. 
The corresponding partition $\lambda = (\lambda_1, \ldots, \lambda_k)$ of $n$ is called the \emph{orbit type} $\Lambda (\alpha)$ of $\alpha$.
We then denote by $\len (\alpha)$ the \emph{length} of this partition, namely ${\len (\alpha) =k}$.  
Consequently, for $\hat{\beta} \in \hat{\cB}$, the dimension $n_{\hat{\beta}}$ of the fixed subspace $(\R^n)^{\Stab (\hat{\beta})}$ is precisely $\len(\hat{\beta})$. 

Furthermore, let $\alpha \in \hat{\cA}$ of orbit type $\Lambda (\alpha) = (\lambda_1, \ldots, \lambda_k)$, and $\beta \in \hat{\cB}$ of orbit type $\Lambda(\beta) = (\mu_1, \ldots, \mu_\ell)$.
Then, the interpretation of ${}_{\hat{\alpha}} (\mathcal{S}_n)_{\hat{\beta}}$ as the inner product of characters gives a combinatorial understanding of the number ${}_{\hat{\alpha}} (\mathcal{S}_n)_{\hat{\beta}}= |\Stab(\hat{\alpha}) \backslash \mathcal{S}_n / \Stab (\hat{\beta}) |$:
it is given by the number $N_{\Lambda(\alpha), \Lambda(\beta)}=|\mathcal{M}_{\Lambda(\alpha), \Lambda(\beta)}|$, where  $\mathcal{M}_{\Lambda(\alpha), \Lambda(\beta)}$ is the set of matrices of size $k \times \ell$ with non-negative integer coefficients such that, for $1 \leqslant i \leqslant k$ the elements of the $i$th row sum up to $\lambda_i$, and for $1 \leqslant j \leqslant \ell$ the elements of the $j$th column sum up to $\mu_j$.
This quantity can be alternatively computed by using the so-called \emph{Kostka numbers} defined for pairs of partitions.
More precisely, we have 
\[
{}_{\hat{\alpha}} (\mathcal{S}_n)_{\hat{\beta}} = N_{\Lambda(\alpha), \Lambda(\beta)}= \sum_{\mu}  K_{\mu, \Lambda(\alpha)} K_{\mu, \Lambda(\beta)},
\]
where $\mu$ runs through the partitions of $n$.
For more details about these interpretations, see \cite[Chapter 7]{Stanley}, in particular Corollary 7.12.3 therein. 

Now we illustrate the potential gain of this reduction, already in a very small example:

	\begin{example}\label{ex:Hugues}Consider the support set $\{ \alpha_0, \ldots, \alpha_7\} = 
		\{(0,0,0)^T, (7,0,0)^T, (0,7,0)^T, $ $(0,0,7)^T, (1,1,2)^T, (1,2,1)^T, (2,1,1)^T, (2,2,2)^T\}$
		and let $G:=\sym_3$ be the symmetric group
		on three elements.
		In order to avoid too heavy notation, we will write $c^{(i)}_j$ instead of $c^{(\alpha_i)}_{\alpha_j}$ and $\nu^{(i)}_j$ instead of $\nu^{(\alpha_i)}_{\alpha_j}$. Consider a signomial \[f(x_1,x_2,x_3) = \sum_{i=0}^7 c_ie^{\langle \alpha_i, (x_1,x_2,x_3) \rangle},\] with $c_0,c_1,c_2,c_3>0$ and $c_4,c_5,c_6,c_7<0$, i.e., set $\mathcal{A} = \{\alpha_0, \ldots, \alpha_3\}$,
		$\mathcal{B} = \{\alpha_4, \ldots, \alpha_7\}$.
		Then $\hat{\cA}=\{\alpha_0,\alpha_1\}$ and 
		$\hat{\cB}=\{\alpha_4,\alpha_7\}$ are sets of orbit representatives.
		The corresponding partitions are $\Lambda(\alpha_0)= \Lambda (\alpha_7) = (3)$, and $\Lambda(\alpha_1)= \Lambda (\alpha_4) = (2,1)$.
		Then, by Corollary~\ref{co:reduce-a}, 
		$f \in C_{\mathrm{SAGE}}(\cA,\cB)$ if and only there exist $c^{(4)} = (c^{(4)}_0,c^{(4)}_1,c^{(4)}_3)$, $\nu^{(4)} = (\nu^{(4)}_0,\nu^{(4)}_1,\nu^{(4)}_3)$, $c^{(7)}=(c^{(7)}_0,c^{(7)}_1)$ and $\nu^{(7)}=(\nu^{(7)}_0,\nu^{(7)}_1)$  satisfying the conditions 
		\begin{align*}
		\nu^{(4)}_0(\alpha_0-\alpha_4) + \nu^{(4)}_1(\alpha_1+\alpha_2-2\alpha_4) + \nu^{(4)}_3(\alpha_3-\alpha_4)& = 0, \\
		\nu^{(7)}_0(\alpha_0-\alpha_7) + \nu^{(7)}_1(\alpha_1+\alpha_2+\alpha_3-3\alpha_7)&=0, \\
		\nu^{(4)}_0 \ln \frac{\nu^{(4)}_0}{c^{(4)}_0} + 2\nu^{(4)}_1 \ln \frac{\nu^{(4)}_1}{c^{(4)}_1}+\nu^{(4)}_3 \ln \frac{\nu^{(4)}_3}{c^{(4)}_3} & \leqslant c_4, \\
		\nu^{(7)}_0 \ln \frac{\nu^{(7)}_0}{c^{(7)}_0}+3 \nu^{(7)}_1 \ln \frac{\nu^{(7)}_1}{c^{(7)}_1} & \leqslant c_7, \\
		3c^{(4)}_0 + c^{(7)}_0  &\leqslant  c_0, \\
		2c^{(4)}_1  + c^{(4)}_3 +  c^{(7)}_1 & \leqslant  c_1.
		\end{align*}

Note that here $\len(\alpha_4) = 2$ and $\len(\alpha_7) = 1$ so that the two vectorial equations bring together $2 + 1$ scalar equations. 	
In total, we get $2(1 + 1 + 2 + 1) = 10$ variables and $ 2+ 1 + 2 +2 = 7$ linear constraints, against $2\cdot 4 \cdot 4 = 32$ variables and $ 4 \cdot 3 + 4 + 4 = 20$ linear constraints forgetting about symmetries. 	
		
	\end{example}
	
In the context of the symmetric group it is natural to consider situations in which the number of variables grows. Such situations were studied for example in the context of sums of squares relaxations. Several examples were observed in which the complexity of the symmetry-adapted semi-definite program is independent of the number of variables.  
Analogously, we now describe natural sequences of signomials where the size of the relative entropy program stabilizes when the number of variables is large enough.

Fix $n_0\in \N$ and start with a signomial $f_{n_0}$ in $n_0$  variables,
represented by the orbit representatives of the exponent vectors $\hat{\cA}$ and $\hat{\cB}$, as well as the corresponding coefficients. 
For each of these exponents $\alpha$, we denote by $\tilde{\Lambda} (\alpha)$ the orbit type of $\alpha$ \emph{where we forget about the $0$ entries}.  
For instance, when $n_0 = 3$, $\hat{\cA} = \{ \hat{\alpha} \} = \{ (1,1,2) \}$ and $\hat{\cB} = \{ \hat{\beta} \} = \{ (0,0,1) \}$, then $\tilde{\Lambda} (\hat{\alpha}) = (2,1)$ while $\tilde{\Lambda} (\hat{\beta}) = (1)$.
Note that these sequences do not have to be partitions of $n$, we therefore introduce 
\[
\wt (\alpha) = \sum_{\lambda \in \tilde{\Lambda}(\alpha)} \lambda,
\]
counting the number of non-zero coordinates of $\alpha$,  and refer to it as the \emph{weight} of $\alpha$. 
Hence $\wt (1,1,2) = 3$, while $\wt(0,0,1) = 1$. 
Now, for every $n \geqslant n_0$, we can see $\alpha$ as an exponent in $\R^n$, by adding $n-n_0$ zeroes. 
This procedure does not affect $\tilde{\Lambda} (\alpha)$ and $\wt (\alpha)$. 
In this way, we can define for every $n > n_0$, the unique $\mathcal{S}_n$-invariant signomial $f_n$ whose support is made of the $\mathcal{S}_n$-orbits of $\hat{\cA}$ and $\hat{\cB}$ with the corresponding coefficients. 
Clearly, in this situation, the number of constraints $C_n$ in Corollary~\ref{co:reduce-a} does not depend on $n$, since it only involves $|\hat{\cB}|$, $|\hat{\cA}|$, and the length of the elements in $\hat{\cB}$, which does not change when $n \geqslant n_0 + 1$.
In this framework, a similar phenomenon holds for the number of variables:

\begin{thm}[\emph{Stabilization Theorem}]\label{thm:stabilization}
Let $n_0 \in \N$, and $\hat{\cA}$, $\hat{\cB}$ be finite orbit representatives of exponent vectors in $\R^{n_0}$.
Consider, for $n\geqslant n_0$, the signomial $f_n$ previously defined, and denote by $V_n$ the number of variables in the symmetry-adapted relative entropy program in Corollary~\ref{co:reduce-a}.
Let 
\[
m = \max \{ \wt (\alpha)\ : \alpha \in \hat{\cA} \cup \hat{\cB}\}.
\] 
Then, for every $n \geqslant 2m$, $V_n = V_{2m}$. 
\end{thm}

\begin{proof}
We shall show by induction that for every $n \geqslant 2m$, $V_n = V_{2m}$. 
The initial step being obvious, assume $n> 2m$. 
The definition of $m$ ensures that for $n \geqslant 2m$, for every $\alpha$ in $\hat{\cA} \cup \hat{\cB}$, the coordinate occurring the most in $\alpha$ is $0$, and therefore 
\[
\Lambda (\alpha) = (n-\wt(\alpha), \lambda_1, \ldots, \lambda_k),
\]
where $(\lambda_1, \ldots, \lambda_k) = \tilde{\Lambda}(\alpha)$. 
Remember that the number of variables is given by
\[
V_n =2 \sum_{\substack{\hat{\alpha}\in \hat{\cA}, \hat{\beta}\in \hat{\cB} } } N_{\Lambda(\hat{\alpha}), \Lambda(\hat{\beta})}^n,
\]
where, if $\Lambda(\hat{\alpha}) = (n - wt(\hat{\alpha}), \lambda_1, \ldots, \lambda_k)$ and $\Lambda(\hat{\beta}) = (n- \wt (\hat{\beta}), \mu_1, \ldots, \mu_\ell)$, the quantity $N_{\Lambda(\hat{\alpha}), \Lambda(\hat{\beta})}^n$ counts the number of matrices of size $(k+1) \times (\ell +1)$ with non-negative integer coefficients of the form
\begin{equation}\label{eq:n}
\begin{blockarray}{cccccc}
n - \wt(\hat{\beta}) & \mu_1 & \mu_2 & \ldots & \mu_\ell \\
\begin{block}{[ccccc]c}
  \cdot & \cdot & \cdot & \ldots & \cdot & n - \wt (\hat{\alpha}) \\
  \cdot & \cdot & \cdot & \ldots & \cdot & \lambda_1 \\
  \vdots & \vdots & \vdots & \ddots & \vdots & \vdots \\
   \cdot & \cdot & \cdot & \ldots & \cdot & \lambda_k \\
\end{block}
\end{blockarray}
\end{equation}
where the labels give the sum of the coefficients in the corresponding row/column. 
Since $\hat{\cA}$ and $\hat{\cB}$ keep the same number of elements, we only need to show that for every $\hat{\alpha} \in \hat{\cA}$, $\hat{\beta}\in \hat{\cB}$,
we have
$N_{\Lambda(\hat{\alpha}), \Lambda(\hat{\beta})}^n = N_{\Lambda(\hat{\alpha}), \Lambda(\hat{\beta})}^{n-1}$.

If we start with a matrix of the form
\begin{equation}\label{eq:n-1}
\begin{blockarray}{cccccc}
n-1 - \wt(\hat{\beta}) & \mu_1 & \mu_2 & \ldots & \mu_\ell \\
\begin{block}{[ccccc]c}
  \cdot & \cdot & \cdot & \ldots & \cdot & n-1 - \wt (\hat{\alpha}) \\
  \cdot & \cdot & \cdot & \ldots & \cdot & \lambda_1 \\
  \vdots & \vdots & \vdots & \ddots & \vdots & \vdots \\
   \cdot & \cdot & \cdot & \ldots & \cdot & \lambda_k \\
\end{block}
\end{blockarray},
\end{equation}
adding $1$ to the top left coefficients provides a matrix of the form \eqref{eq:n}, which proves ${N_{\Lambda(\hat{\alpha}), \Lambda(\hat{\beta})}^n \geqslant N_{\Lambda(\hat{\alpha}), \Lambda(\hat{\beta})}^{n-1}}$.

In order to show the reverse inequality, we claim that the top left coefficient in \eqref{eq:n} cannot be $0$. Indeed, if it is $0$, then the sum of the coefficients in the first row is at most 
\[
\mu_1 + \mu_2 + \cdots + \mu_\ell = \wt(\hat{\beta}).
\]
This implies $n - \wt (\hat{\alpha}) \leqslant \wt (\hat{\beta})$, which forces $n \leqslant 2m$ and gives a contradiction. 

Now, if the top left coefficient is a positive integer, subtracting $1$ to this coefficient provides a matrix of the form \eqref{eq:n-1}, which proves $N_{\Lambda(\hat{\alpha}), \Lambda(\hat{\beta})}^n \leqslant N_{\Lambda(\hat{\alpha}), \Lambda(\hat{\beta})}^{n-1}$, and hence ${V_n = V_{n-1}}$. 

\end{proof} 
Building on this, we can actually show that for a large class of problems we have a stabilization.
\begin{thm}
Let $k,l,w \in \N$ be fixed. Then for every integer $n \geqslant 2w$ and every $\mathcal{S}_n$-invariant signomial $f \in \mathcal{C}(\mathcal{A},\mathcal{B})$ with $|\hat{\cA}| \leqslant k$, $|\hat{\cB}| \leqslant l$, and \[\max_{\hat{\gamma} \in \hat{\cA} \cup \hat{\cB}} \wt(\hat{\gamma}) \leqslant w,\] the number of constraints and the number of variables  of the symmetry adapted program are bounded by constants only depending on $k$, $l$ and $w$: \[C_n \leqslant k+l +l (w+1)
\; \text{ and } \; V_n \leqslant 2lk u(w),\] where 
$\displaystyle u(w) = \sum_{i=0}^w \binom{w}{i}^2 i! \, .$
\end{thm}

\begin{proof} Let us begin with the number of constraints. This  follows from Theorem~\ref{thm:stabilization}, because  $|\hat{\cA}| \leqslant k$, $|\hat{\cB}| \leqslant l$ and $n_{\hat{\beta}}\leqslant w+1$, since $\wt(\hbeta)\leqslant w.$ As in the previous proof, we have $\Lambda(\hat{\alpha})_1 = n-\wt(\hat{\alpha})$ and similarly for $\hat{\beta}$. For the number of variables, we will show that  \begin{equation}\label{eq:dlc}N_{\Lambda(\hat{\alpha}),\Lambda(\hat{\beta})}\leqslant N_{(n-w,1^w),(n-w,1^w)}= u(w)\end{equation} for every $\hat{\alpha},\hat {\beta}$ satisfying the conditions of the theorem. This will be done in two steps.  First, we show that if $\lambda=(\lambda_1,\ldots, \lambda_t,1)$ is a partition, and  $\lambda'=(\lambda_1,\lambda_2,\ldots,\lambda_{t-1},\lambda_t+1)$, then for every partition $\mu$, we have \[N_{\lambda',\mu} \geqslant N_{\lambda,\mu} \quad \textrm{and} \quad N_{\mu, \lambda'} \geqslant N_{\mu,\lambda}.\] Indeed, there is a surjection from the set $\mathcal{M}_{\lambda',\mu}$ onto $\mathcal{M}_{\lambda,\mu}$. Namely let  $(x_1,\ldots,x_k)$ denote the $t$-th line of an element in $\mathcal{M}_{\lambda,\mu}$. Let $s$ be such that $x_s >0$. Replacing the $t$-th line of this element by
 $(x_1,\ldots,x_{s-1},x_s-1,x_{s+1},\ldots x_k)$ 
 and inserting $(0,\ldots,0,1,0,\ldots,0)$ as the $(t+1)$-th line we get  an element in $\mathcal{ N}_{\lambda',\mu}$. By applying this procedure recursively for rows and columns we get the inequality in~\eqref{eq:dlc}.

To show the equality in~\eqref{eq:dlc}, observe that the top-left element of $N_{(n-w,1^w),(n-w,1^w)}$has to be an integer $k$ between $n-2w$ and $n-w$. For every such choice, we have to distribute $n-w-k$ ones in the first row and first column. This gives $\binom{w}{n-w-k}^2$ possibilities. Restricted to the $w \times w$ lower right submatrix, these selected lines and columns contain only $0$. For each of these possibilities, after removing these chosen lines and columns, we get an $(n-k) \times (n-k)$ matrix which contains exactly one $1$ per line and column. There are $(n-k)!$ such matrices. By a change of the index variable, we get the desired result:
\begin{align*}
	V_n= 2 \sum_{\substack{  \hat{\alpha} \in \hat{\cA} \\ \hat{\beta} \in \hat{\cB} }} {}_{\hat{\alpha}} (\mathcal{S}_n)_{\hat{\beta}}
	 = 2\sum_{\substack{  \hat{\alpha} \in \hat{\cA} \\ \hat{\beta} \in \hat{\cB} }}  N_{\Lambda(\alpha), \Lambda(\beta)} 
	 \leqslant    2lk u(w).
\end{align*}
\end{proof}

We conclude this section by giving explicit estimates on the  signomials where ${|\hat{\cB}|=1}$, and $\hat{\cA} = \{ 0, \hat{\alpha}\}$. We have chosen four different classes of examples that show the influence of the sizes of the orbits on the numbers of variables and constraints. 
These classes represent extremal situations, namely when the orbits are either very large or very small. 
In these situations, we can actually compute the exact number of variables and constraints in both cases according to the previous discussions. 
Note that the last case falls into the framework of Theorem~\ref{thm:stabilization}, where $\wt (\hat{\alpha}) = \wt (\hat{\beta}) = 1$. 
There is also a stabilization in the first sequence: when $\len (\hat{\beta}) = 1$, for every $\hat{\alpha} \in \hat{\cA}$, the number of variables ${}_{\hat{\alpha}} (\mathcal{S}_n)_{\hat{\beta}}$ is equal to $1$.
The subsequent table summarizes our analysis. Specific signomials realizing the cases
are given in Examples~\ref{ex:1}--\ref{ex:4} in the next section. 
\begin{table}[H]
\begin{center}
\begin{tabular}{|c|c|c|c|c|c|c|}
  \cline{3-6}
                \multicolumn{1}{c}{}  &\multicolumn{1}{c|}{} &\multicolumn{2}{c|}{\begin{bfseries}Standard method\end{bfseries}} & \multicolumn{2}{c|}{\begin{bfseries}Symmetric method\end{bfseries}}\\
                \hline
$|\mathcal{S}_n\cdot \hat{\beta}|$   & $|\mathcal{S}_n\cdot \hat{\alpha}|$ &  $V_n$& $C_n$ & $V_n$& $C_n$ & Example \\
\hline
$1$  & $n!$ &  $2n! + 3$& $n!+n+2$& $5$ & $4$ & \ref{ex:1}\\
\hline
$n!$  & $n$ &  $2(n+1)n! +1$& $(n+1)(n!+1)$& $2n+3$ & $n+3$ & \ref{ex:2}\\
\hline
$n!$  & $n!$ &  $2(n!+1)n! + 1$& $n!(n+2)+1$& $2n!+3$ & $n+3$ & \ref{ex:3}\\
\hline
$n$  & $n$ &  $2n(n+1)+1$& $(n+1)^2$& $7$ & $5$ & \ref{ex:4}\\
\hline
\end{tabular}
\end{center}
\caption{Comparison of the parameters when $\hat{\cA}=\{0,\hat{\alpha}\}$ and $\hat{\cB}=\{\hat{\beta}\}$.}\label{tab:recap}
\end{table}

\section{Numerical experiments\label{se:computations}}

To illustrate the previous considerations, we present in this section classes of examples that spotlight the computational gains by the comparison of calculation times in the case of the symmetric group. 
For these computations, we used the MOSEK 
solver and Python 3.7 on an Intel(R) Xeon(R) Platinum 8168 CPU with 2.7 GHz and 768 GB of RAM under CentOS Linux release 7.9.2009. Keeping the previous notation, for the standard method, that is the method that does not exploit the symmetries, the input consists of $\cA$, $\cB$ as well as the coefficients, while for the symmetry-adapted version (Corollary~\ref{co:reduce-a}), the input is $\hat{\cA}$, $\hat{\cB}$ and the coefficients. This difference of input is mainly due to practical considerations and does not in itself influence the comparison of the time used by the solver. When both methods give an answer, the bounds coincide.

In all the tables in the sequel, $\mathrm{dim}$ is the dimension, $V_n$ and $C_n$  are the number of variables and constraints of the program, while $t_s$ and $t_r$ denote the solver time and the overall running time (including the building of the optimization program) in seconds. These examples confirm that symmetry reduction can drastically decrease computation complexity, which results in faster 
computation and the possibility of solving larger problems.

The first four examples give numerical results for each of the classes discussed in Table~\ref{tab:recap}.
We tried to choose the coefficients in a way that avoids numerical issues, namely preventing the bound from being either too small or too large. 
We conducted the experiments until the number of inner or outer coefficients was too large to deal 
with in the standard program. The program using the symmetry reduction
can even successfully carry out the experiments in dimensions beyond the ones 
listed in the tables. When further increasing dimensions, one has to be 
careful though, because 
some of the coefficients in the symmetry reduced constraints may endure a combinatorial explosion, which might cause some numerical issues.
% This could be coped with with a rescaling of some variables.

\begin{example}\label{ex:1}

Consider first the signomial \[f^{(1)}_n= n\sum_{\sigma\in\mathcal{S}_n} \sigma \exp(\langle \alpha, x \rangle)-n\exp(\langle \beta, x \rangle),\] 
where $\beta=(1,\ldots,1)$ and $\alpha=(1,2,\ldots,n)$.          The numerical results are shown in Table~\ref{tab:exa64}.

\begin{table}[H]
   \begin{center}
               \begin{tabular}{||c |c || c|c|S|S||c|c|c|c||}
                        \cline{3-10}
                \multicolumn{1}{c}{}  &\multicolumn{1}{c||}{} &\multicolumn{4}{c||}{\begin{bfseries}Standard method\end{bfseries}} & \multicolumn{4}{c||}{\begin{bfseries}Symmetric method\end{bfseries}}\\
                        \hline
                        $\mathrm{dim}$ & {bound} & $V_n$ & $C_n$ & $t_s$ & $t_r$ & $V_n$ & $C_n$ & $t_s$ & $t_r$ \\
                        \hline
2 & -0.0741 &7 & 6 & 0.0458 & 0.0472   & 5 & 4 & 0.0479 & 0.0496 \\
\hline
3 & -0.1250  &15 & 11   & 0.0666 & 0.0689& 5 & 4 &  0.0459 & 0.0471 \\
\hline
4 & -0.1566 &51 & 30   &0.0690 & 0.0732 & 5 & 4 &  0.0679 & 0.0692  \\
\hline
5 & -0.1757 &243& 127  & 0.2005 & 0.2386 & 5 & 4 &  0.0641 & 0.0654 \\
\hline
6 & -0.1868 &1443&728   & 1.045 & 1.174 & 5 & 4 &  0.0270 & 0.0280 \\
\hline
7 & -0.1929 & 10083&5049  &14.84 & 15.67 & 5 & 4 & 0.0268 & 0.0278   \\
\hline
8 & -0.1956 & 80643 &40330  & 236.2 & 242.7& 5 & 4 &  0.0274 & 0.0283 \\
\hline
9 & -0.1962  & 725763&362891& 24774 & 24837& 5 & 4 &  0.0322 & 0.0332 \\
\hline
                        \end{tabular}
\end{center}\caption{Numerical results for $f^{(1)}_n$.}\label{tab:exa64}	\end{table}

\end{example}

\begin{example}\label{ex:2}

Consider now the signomial \[f^{(2)}_n= \sum_{i=1}^{n}  \exp(n^2 x_i)-\frac{1}{(n-1)!}\sum_{\sigma \in \mathcal{S}_n}\sigma \exp(\langle \beta, x \rangle),\] where $\beta=(1,2,\ldots,n)$ (and $\alpha=(n^2,0,\ldots,0)$).          The numerical results are shown in Table~\ref{tab:exa62}.
        
\begin{table}[H]
            \begin{center}
                \begin{tabular}{||c |c || c|c|S|S||c|c|c|c||}
                        \cline{3-10}
                \multicolumn{1}{c}{}  &\multicolumn{1}{c||}{} &\multicolumn{4}{c||}{\begin{bfseries}Standard method\end{bfseries}} & \multicolumn{4}{c||}{\begin{bfseries}Symmetric method\end{bfseries}}\\
                        \hline 
                        $\mathrm{dim}$ & {bound} & $V_n$ & $C_n$ & $t_s$ & $t_r$ & $V_n$ & $C_n$ & $t_s$ & $t_r$ \\
                        \hline
                        2 & -0.2109 &13 & 9 &  0.0568 & 0.0585 & 7 & 5 & 0.0583 & 0.060 \\
                        \hline
                        3 &-0.4444 &49 & 28& 0.0893 & 0.0938 & 9 & 6& 0.0438 & 0.046 \\
                        \hline
                        4&-0.6853 &241 & 125 & 0.2670 & 0.2830  & 11 & 7& 0.0679 & 0.0717  \\
                        \hline
                        5&-0.9295 &1441 & 726&1.097 & 1.183 &13 & 8& 0.1028 & 0.1084\\
                        \hline
                        6& -1.176 &10081 & 5047&8.276 & 8.866 &15 & 9& 0.1399 & 0.1605\\
                        \hline
                        7&-1.423 &80641 & 40328&179.3 & 191.3 & 17 & 10&0.0781 & 0.0857 \\
                        \hline
                        8& -1.670& 725761 & 362889&16709 & 17744& 19 & 11& 0.1020 & 0.1111  \\
                        \hline
                \end{tabular}
        \end{center}
        \caption{Numerical results for $f^{(2)}_n$.}\label{tab:exa62}\end{table}

\end{example}

\begin{example}\label{ex:3}
        
        Next, we consider the case where both orbits are of maximal size. Let \[f^{(3)}_n= \frac{1}{n!}\sum_{\sigma \in \mathcal{S}_n} \sigma \exp(\langle \alpha, x \rangle)-\frac{1}{n!}\sum_{\sigma \in \mathcal{S}_n}\sigma \exp(\langle \beta, x \rangle),\] 
where $\beta=(1,2,\ldots,n)$ and $\alpha=(2,8,\ldots,2n^2)$.

         The numerical results are shown in Table~\ref{tab:exa63}.
\begin{table}[H]        
        	             \begin{center}
                \begin{tabular}{||c |c || c|c|S|S||c|c|c|c||}
                        \cline{3-10}
                \multicolumn{1}{c}{}  &\multicolumn{1}{c||}{} &\multicolumn{4}{c||}{\begin{bfseries}Standard method\end{bfseries}} & \multicolumn{4}{c||}{\begin{bfseries}Symmetric method\end{bfseries}}\\
                        \hline
                        $\mathrm{dim}$ & {bound} & $V_n$ & $C_n$ & $t_s$ & $t_r$ & $V_n$ & $C_n$ & $t_s$ & $t_r$ \\
                        \hline 
2 & -0.4178 &13 & 9 &0.0852 & 0.0873 & 7 & 5 & 0.0741 & 0.0766 \\
\hline
3 & -0.5162 &85 & 31 &0.0674 & 0.0751  & 15 & 6 &0.0598 & 0.0629 \\
\hline
4 &-0.5824& 1201	& 145 & 0.3074 & 0.3466 & 51 & 7 & 0.2249 & 0.2405\\
\hline
5& -0.6305 &29041 & 841 & 12.81 & 13.51 & 243 & 8 & 0.7791 & 0.8642\\
\hline
6& -0.6675 &1038241 & 5761 & 14649 & 14688  & 1443 & 9 &4.303 & 4.818 \\
\hline
                        \end{tabular}
\end{center}	
\caption{Numerical results for $f^{(3)}_n$.} \label{tab:exa63} \end{table}

\end{example}

\begin{example}\label{ex:4}

Finally, we consider the case where both orbits are small. Let 
\[f^{(4)}_n= \sum_{i=1}^{n}  \exp(n^2 x_i)-\sum_{i=1}^{n} \exp({(n-1)(x_1+\cdots+x_n) + x_i}),\] 
($\beta=(n,n-1,n-1,\ldots,n-1)$ and $\alpha=(n^2,0,\ldots,0)$).          The numerical results are shown in Table~\ref{tab:exa65}. 

\begin{table}[H]        
        	             \begin{center}
                \begin{tabular}{||c |c || c|c|S|S||c|c|c|c||}
                        \cline{3-10}
                \multicolumn{1}{c}{}  &\multicolumn{1}{c||}{} &\multicolumn{4}{c||}{\begin{bfseries}Standard method\end{bfseries}} & \multicolumn{4}{c||}{\begin{bfseries}Symmetric method\end{bfseries}}\\
                        \hline
                        $\mathrm{dim}$ & {bound} & $V_n$ & $C_n$ & $t_s$ & $t_r$ & $V_n$ & $C_n$ & $t_s$ & $t_r$ \\
                        \hline
10 & -0.3468& 221& 121& 0.2621& 0.2772& 7 &5 &0.0543& 0.0552 \\ \hline
20 & -0.3580& 841& 441& 0.2712& 0.2924& 7 &5& 0.0381& 0.0392\\\hline
30& -0.3615& 1861& 961& 0.5142& 0.5575& 7 &5 &0.0375 &0.0386\\\hline
40 & -0.3631& 3281& 1681& 1.123& 1.189& 7 &5 &0.0381 &0.0394\\\hline
50 & -0.3641& 5101& 2601& 2.356& 2.466& 7 &5 &0.0371 &0.0385\\\hline
100 & -0.3660& 20201& 10201& 46.22& 46.63& 7 &5 &0.0365 &0.0384\\\hline
150 & -0.3667& 45301& 22801& 303.2& 304.3& 7 &5& 0.0454 &0.0481\\\hline
200& -0.3670& 80401& 40401& 1530& 1532& 7 &5 &0.0565 &0.0590\\\hline
250& -0.3671& 125501& 63001& 5358& 5361& 7 &5 &0.0335 &0.0362\\\hline
300& -0.3673& 180601& 90601 &12990& 12 995& 7 &5 &0.0476 &0.0509\\\hline
350& -0.3674& 245701& 123201 &30 219& 30 226& 7 &5& 0.0513 &0.0551\\\hline
\end{tabular}
\end{center}	
\caption{Numerical results for $f^{(4)}_n$.} \label{tab:exa65} \end{table}

\end{example}

\begin{example}\label{ex:exa666}
Next, we  give an example where $\cA$ and $\cB$ consist of two orbits each: 
\[\hat{\cA} = \{ (n^2, 0, \ldots, 0), (1, 4, \ldots, n^2) \} \; \text{ and } \;
\hat{\cB} = \{ (1, \ldots, 1), (1, 2, \ldots, n) \}. \]

In this case, we are still able to compute the number of constraints and the number of variables. With the standard approach, 
\[
V_n = 2(n! + n +1)(n!+1)+1, \quad C_n=(n!+1)(n+2) + n, 
\]
while using symmetries,
\[
V_n = 2n! + 2n + 9, \quad C_n=n+6. 
\]
Table~\ref{tab:exa66} shows the numerical results for the signomials
\[
g_n= \frac{1}{n}\sum_{i=1}^n \exp({n^2x_i}) + \frac{1}{n!} \sum_{\sigma \in \mathcal{S}_n} \sigma \exp(\langle \alpha, x \rangle)- \exp({x_1 + \dots + x_n}) - \frac{1}{n!} \sum_{\sigma \in \mathcal{S}_n} \sigma \exp( \langle \beta, x \rangle)
\]
for $\alpha = (1, 4, \ldots, n^2)$ and $\beta = (1, 2, \ldots, n)$.

\begin{table}[H]        
        	             \begin{center}
                \begin{tabular}{||c |c || c|c|S|S||c|c|c|c||}
                        \cline{3-10}
                \multicolumn{1}{c}{}  &\multicolumn{1}{c||}{} &\multicolumn{4}{c||}{\begin{bfseries}Standard method\end{bfseries}} & \multicolumn{4}{c||}{\begin{bfseries}Symmetric method\end{bfseries}}\\
                        \hline
                        $\mathrm{dim}$ & {bound} & $V_n$ & $C_n$ & $t_s$ & $t_r$ & $V_n$ & $C_n$ & $t_s$ & $t_r$ \\
                        \hline 
2 & -0.4311 &31 & 14 & 0.1003 & 0.1027& 17 & 8 & 0.1290 & 0.1326\\
\hline
3 &-0.6643  &141 & 38 &0.1232 & 0.1291 & 27 & 9 & 0.1285 & 0.1375\\
\hline
4 & -0.8070 &1451 & 154 &0.3944 & 0.4390  & 65 & 10 & 0.2275 & 0.2470\\
\hline
5 &-0.9009  & 30493& 852 &14.88 & 15.55 & 259 & 11 & 0.6945 & 0.7689\\
\hline
6 & -0.9660 & 1048335& 5774 & 14187 & 14220 & 1461 & 12 &3.832 & 4.281\\
\hline
\end{tabular}
\end{center}	
\caption{Numerical results for $g_n$.} \label{tab:exa66} \end{table}
\end{example}

\begin{example} Finally, we consider the symmetric dwarfed signomial. 
 The dwarfed polynomial is available  on the accompanying web page of~\cite[Example~4.5]{de-wolff-dwarfed-2018}.
 This is a polynomial in seven variables, of total degree six, and degree four in each variable. We symmetrize it by applying the Reynolds operator with respect to the full symmetric group $\mathcal{S}_7$. The associated signomial (the symmetrized dwarfed signomial) has $113$ non-zero coefficients. We have
 \[\hat{\mathcal{A}} = \{(4,2,0,0,0,0,0), (4,0,0,0,0,0,0),(2,2,2,0,0,0,0)\}\] and \[\hat{\mathcal{B}} = \{(2,2,0,0,0,0,0),(2,0,0,0,0,0,0)\}.\] The symmetry reduced program needs $0.191$ seconds to find the bound, while the standard program needs $0.912$ seconds.
\end{example}

	\section{Conclusion and open questions\label{se:conclusion}}
	
	We have developed techniques to exploit symmetries in AM/GM-based optimization and confirmed
	their benefit in terms of computational results. In particular, in the case of symmetric signomials, we showed that both theoretically as well as practically our orbit reduction allow for  substantial computational gains. This motivates a theoretical  study of the strength of the AM/GM bounds in this framework. In particular, it  encourages the comparison of the symmetric SAGE cones with respect to the cone of symmetric non-negative signomials. 

The orbit decomposition in Theorem~\ref{th:symmetric-decomp} can also be used
  for SAGE and SONC decompositions obtained by second-order cone 
  representations, which were studied \cite{averkov-2019,naumann-theobald-2021-second-order,wang-magron-2020}.
  It would be interesting to investigate symmetry exploitation in that context in 
  further detail.

	%\printbibliography[heading=bibintoc]
\end{document}